\documentclass[11pt]{article}

\usepackage{amsfonts,amsmath,latexsym,color,epsfig,hyperref,mathdots,tikz,centernot}
\newtheorem{theorem}{Theorem}[section]
\newtheorem{lemma}{Lemma}[section]

\newtheorem{question}{Question}[section]

\newenvironment{proof}
      {\medskip\noindent{\bf Proof:}\hspace{1mm}}
      {\hfill$\Box$\medskip}
\def\qed{\ifvmode\mbox{ }\else\unskip\fi\hskip 1em plus 10fill$\Box$}

\input{epsf}

\makeatletter
\def\Ddots{\mathinner{\mkern1mu\raise\p@
\vbox{\kern7\p@\hbox{.}}\mkern2mu
\raise4\p@\hbox{.}\mkern2mu\raise7\p@\hbox{.}\mkern1mu}}
\makeatother

\title{\vspace{-0.7cm}Lines in Euclidean Ramsey theory}
\author{David Conlon\thanks{Mathematical Institute, Oxford OX2 6GG,
United Kingdom. Email: {\tt david.conlon@maths.ox.ac.uk}. Research
supported by a Royal Society University Research Fellowship and by ERC Starting Grant 676632.}\and
Jacob Fox\thanks{Department of Mathematics, Stanford University, Stanford, CA 94305, USA. Email: {\tt jacobfox@stanford.edu}. Research supported by a Packard Fellowship and by NSF Career Award DMS-1352121.}}
\date{}

\begin{document}
\maketitle

\begin{abstract}
Let $\ell_m$ be a sequence of $m$ points on a line with consecutive points of distance one. For every natural number $n$, we prove the existence of a red/blue-coloring of $\mathbb{E}^n$ containing no red copy of $\ell_2$ and no blue copy of $\ell_m$ for any $m \geq 2^{cn}$. This is best possible up to the constant $c$ in the exponent. It also answers a question of Erd\H{o}s, Graham, Montgomery, Rothschild, Spencer and Straus from 1973. They asked if, for every natural number $n$, there is a set $K \subset \mathbb{E}^1$ and a red/blue-coloring of $\mathbb{E}^n$ containing no red copy of $\ell_2$ and no blue copy of $K$.
\end{abstract}

\section{Introduction}

Let $\mathbb{E}^n$ denote $n$-dimensional Euclidean space, that is, $\mathbb{R}^n$ equipped with the Euclidean distance. Following Erd\H{o}s, Graham, Montgomery, Rothschild, Spencer and Straus~\cite{EGMRSS2}, we study the following question. 

\begin{question}
For which subsets $K \subset \mathbb{E}^n$ does every red/blue-coloring of $\mathbb{E}^n$ contain a red pair of points of distance one or a blue isometric copy of $K$?
\end{question}

In what follows, we will write $\ell_m$ for a sequence of $m$ points on a line with consecutive points of distance one and $\mathbb{E}^n \longrightarrow (\ell_2,K)$ if every red/blue-coloring of $\mathbb{E}^n$ contains either a red copy of $\ell_2$ or a blue copy of $K$, where a copy of a set will always mean an isometric copy. Conversely, $\mathbb{E}^n \centernot\longrightarrow (\ell_2,K)$ expresses the fact that there is some red/blue-coloring of $\mathbb{E}^n$ which contains neither a red copy of $\ell_2$ nor a blue copy of $K$. 

The problem of determining which $n$ and $K$ satisfy the relation $\mathbb{E}^n \longrightarrow (\ell_2,K)$ has received considerable attention, with a particular focus on small values of $n$. For example, Erd\H{o}s et al.~\cite{EGMRSS2} showed that $\mathbb{E}^2 \longrightarrow (\ell_2,\ell_4)$ and $\mathbb{E}^2 \longrightarrow (\ell_2,K)$ for any three-point set $K$. Juh\'asz~\cite{Juh} later improved the latter result to cover all four-point planar sets, while just recently Tsaturian~\cite{Tsat} improved the former result by showing that $\mathbb{E}^2 \longrightarrow (\ell_2,\ell_5)$. In three dimensions, Iv\'an~\cite{Ivan} showed that $\mathbb{E}^3 \longrightarrow (\ell_2,K)$ for any five-point set $K \subset \mathbb{E}^3$. The particular case where $K = \ell_5$ was recently improved by Arman and Tsaturian~\cite{ArmTsat}, who showed that $\mathbb{E}^3 \longrightarrow (\ell_2,\ell_6)$.

On the other hand, Csizmadia and T\'oth~\cite{CsTo} identified a set $K$ of $8$ points in the plane, namely, a regular heptagon with its center, such that $\mathbb{E}^2 \centernot\longrightarrow (\ell_2,K)$. This improved a result of Juh\'asz~\cite{Juh}, who had previously identified a set $K$ of $12$ points with the same property. Our chief concern in this paper will be with extending these results to higher dimensions by studying the smallest possible size of a set $K \subset \mathbb{E}^n$ such that $\mathbb{E}^n \centernot\longrightarrow (\ell_2,K)$.

In general, $|K|$ can be unbounded in terms of $n$ and still satisfy $\mathbb{E}^n \longrightarrow (\ell_2,K)$. For example, any subset $K$ of the unit sphere in $\mathbb{E}^n$ satisfies $\mathbb{E}^n \longrightarrow (\ell_2,K)$. Indeed, in a red/blue-coloring of $\mathbb{E}^n$, if there is no red point, then we clearly get a copy of $K$, while if there is a red point, then the sphere of radius one around that point must be blue, so we again get a blue copy of $K$. 

However, our main result shows that under some mild conditions a set $K \subset \mathbb{E}^n$ such that $\mathbb{E}^n \longrightarrow (\ell_2,K)$ can have size at most exponential in $n$. To state the result, we say that a point set $S \subset \mathbb{E}^n$ is {\it $t$-separated} if any two points in $S$  have distance at least $t$. Here and throughout, we use $\log$ to denote log base $2$. 

\begin{theorem} \label{thm:main}
If $R>2$ and $K$ is a $1$-separated set of points in $\mathbb{E}^n$ with diameter at most $R - 1$ and $|K| > 10^{4n} \log R$, then $\mathbb{E}^n \centernot \longrightarrow (\ell_2,K)$. 
\end{theorem}

In particular, for $m = 10^{5n}$, we see that $\mathbb{E}^n \centernot\longrightarrow (\ell_2, \ell_m)$. This simple corollary is already enough to answer a problem raised by Erd\H{o}s et al.~\cite{EGMRSS2}, namely, whether, for every natural number $d$, there is a natural number $n$ depending only on $d$ such that $\mathbb{E}^n \rightarrow (\ell_2, K)$ for every $K \subset \mathbb{E}^d$. Erd\H{o}s et al.~state that they expect the answer to this question to be negative and our result confirms this already for $d = 1$, a special case stressed in~\cite{EGMRSS2}, showing that $n$ must grow logarithmically in the size of $|K|$.

The exponential dependence in Theorem~\ref{thm:main}, and hence the logarithmic dependence above, is also necessary. In fact, Szlam~\cite{Sz01} proved the stronger result that every red/blue-coloring of $\mathbb{E}^n$ contains either a red copy of $\ell_2$ or a blue \emph{translate} of any set $K$ of size at most $2^{c'n}$.
For the sake of completeness, we include his short proof here. We will need the seminal result of Frankl and Wilson~\cite{FW} that there exists a positive constant $c'$ such that any coloring of $\mathbb{E}^n$ with at most $2^{c'n}$ colors contains a pair of points of distance one with the same color (see~\cite{R00} for the current best estimate on $c'$). 

Suppose now that $K = \{k_1, \dots, k_t\} \subset \mathbb{E}^n$ is a set of size at most $2^{c' n}$ and there is a red/blue-coloring of $\mathbb{E}^n$ with no blue copy of $K$. Then, for each $p \in \mathbb{E}^n$, there is at least one $i$ such that $p + k_i$ is red, since otherwise the set $p + K$ would be a blue translate of $K$. We may therefore color the points of $\mathbb{E}^n$ in $t \leq 2^{c'n}$ colors, giving the point $p$ the color $i$ for some $i$ such that $p + k_i$ is red, always choosing the minimum such $i$. By the result of Frankl and Wilson, there must then exist two points $p$ and $p'$ of distance one which receive the same color, say $j$. But then $p + k_j$ and $p' + k_j$ are two points of distance one both of which are colored red. This gives the required result. In particular, we have the following counterpart to Theorem~\ref{thm:main}, which we again stress is due to Szlam~\cite{Sz01}.

\begin{theorem} \label{thm:FW}
There exists a positive constant $c'$ such that $\mathbb{E}^n \longrightarrow (\ell_2,K)$ for any set $K \subset \mathbb{E}^n$ of size at most $2^{c'n}$.
\end{theorem}

\section{Proof of Theorem~\ref{thm:main}}

We will prove the existence of a periodic red/blue-coloring of $\mathbb{E}^n$ (with period $R$ in the standard coordinates) such that no two red points have distance one and there is no blue copy of $K$.

Let $\mathbb{T}_R^n=(\mathbb{E}/R \mathbb{Z})^n$ be the $n$-dimensional torus with period $R$ in each direction. Let $P$ be any maximal $1/3$-separated subset of $\mathbb{T}_R^n$. One can simply construct such a set $P$ greedily. Consider the Voronoi decomposition of $\mathbb{T}_R^n$ with respect to $P$. This partitions $\mathbb{T}_R^n$ into cells $V_p$, one for each point $p \in P$, where $V_p$ consists of the set of points closer to $p$ than any other point in $P$. From the maximality of $P$, every point in $V_p$ has distance at most $1/3$ from $p$. In particular, each $V_p$ has diameter at most $2/3$.  

\begin{lemma}
$|P| \leq (4n^{1/2}R)^n$. 
\end{lemma}

\begin{proof}
Since each pair of points in $P$ have distance at least $1/3$, the balls of radius $r=1/6$ around each point are disjoint. A ball in $n$-space of radius $r$ has volume $r^n \pi^{n/2}/\Gamma(n/2 + 1)$, where the gamma function satisfies $\Gamma(n/2+1)=(n/2)!$ if $n$ is even and $\Gamma(n/2+1)=\sqrt{\pi} \cdot n!!/2^{(n+1)/2}$ if $n$ is odd. In either case, we have $\Gamma(n/2 + 1) \leq n^{n/2}$, so the volume of the $n$-dimensional ball is at least $(r^2\pi/n)^{n/2}$. The balls of radius $1/6$ around the points of $P$ are disjoint and the volume of the torus $\mathbb{T}_R^n$ is $R^n$, so the number of points in $P$ is at most $R^n/((1/6)^2\pi/n)^{n/2} = (36nR^2/\pi)^{n/2} < (4n^{1/2}R)^n$.
\end{proof}

\begin{lemma}\label{basic}
 If $S \subset \mathbb{E}^n$ is $t$-separated, then, for any point $p \in \mathbb{E}^n$ and any $s \geq 0$, the number of points of $S$ within distance $s$ of $p$ is at most $(2s/t + 1)^n$. 
\end{lemma}

\begin{proof}
 The balls of radius $t/2$ around each point of $S$ are disjoint and, for each point $p' \in S$ with distance at most $s$ from $p$, the ball of radius $s+t/2$ around $p$ contains the ball of radius $t/2$ around $p'$. Hence, by a volume argument, there are at most $(\frac{s+t/2}{t/2})^n = (2s/t + 1)^n$ points of distance at most $s$ from $p$. 
\end{proof}

\begin{lemma}\label{halfspace}
Each copy in $\mathbb{E}^n$ of the Voronoi cell $V_p$ is a convex body defined by the intersection of at most $5^n$ half-spaces. 
\end{lemma}

\begin{proof}
A point $q$ on the boundary of $V_p$ is on the hyperplane equidistant to $p$ and some other point $p' \in P$, where this distance is at most $1/3$. This implies that $p'$ has distance at most 2/3 from $p$. Since the points in $P$ are $1/3$-separated, Lemma \ref{basic} implies that there are at most $5^n$ points of $P$ of distance at most $2/3$ from $p$. Therefore, since the Voronoi cell $V_p$ is the intersection of half-spaces that are defined by hyperplanes which are equidistant from $p$ and $p'$ for some $p'$ of distance at most $2/3$ from $p$, the result follows.
\end{proof}

\begin{lemma}\label{sep}
If $K$ is a $1$-separated point set in $\mathbb{E}^n$ and $s \geq 1$, then there is a set $K' \subset K$ that is $s$-separated and has size at least $|K|/(2s+1)^n$. 
\end{lemma}

\begin{proof}
By Lemma \ref{basic} with $t=1$, for each point $p$, there are at most $(2s+1)^n$ points of $K$ within distance $s$ of $p$ (including $p$ itself). We can then greedily construct the set $K'$, getting at least one point in $K'$ for every $(2s+1)^n$ points from $K$, giving the desired bound.
\end{proof}

Let $Q$ be a random subset of $P$ formed by picking each point in $P$ with probability $x=20^{-n}$ independently of the other points. Let $S$ be the subset of $Q$ where $s \in S$ if and only if there is no other point $s' \in Q$ of distance at most $5/3$ from $s$. By Lemma \ref{basic}, there are at most $(2(5/3)/(1/3)+1)^n = 11^n$ points of $P$ of distance at most $5/3$ from any point. For a given point $p \in P$, the probability that $p \in S$ is therefore at least $x(1-x)^{11^n} > x/2$ as $x=20^{-n}$. 

Let $V_1,\ldots,V_m$ be the Voronoi cells of points in $S$. We will color the points in these Voronoi cells red, including the boundaries, and everything else blue. We consider the periodic coloring of $\mathbb{E}^n$ given by the coloring of $\mathbb{T}_R^n$. Observe that there is a pair of red points of distance one in $\mathbb{T}_R^n$ if and only if there is a pair of red points of distance one in the periodic coloring of $\mathbb{E}^n$ and  there is a blue copy of $K$ in $\mathbb{T}_R^n$ if and only if there is a blue copy of $K$ in $\mathbb{E}^n$.

We first claim that there are no two red points $q$ and $q'$ at distance one. Indeed, if $q$ and $q'$ are in the same Voronoi cell, then, as the diameter of each Voronoi cell is at most $2/3$, we have a contradiction. If $q$ and $q'$ are in copies of the same cell in the periodic tiling, then their distance is at least $R-2/3 > 1$. If $q$ and $q'$ are in different cells, with $q  \in V_p$ and $q' \in V_{p'}$, then, since $q$ has distance at most $1/3$ from $p$ and $q'$ has distance at most $1/3$ from $p'$, $p$ and $p'$ have distance at most $5/3$. However, by construction, if $p \in S$, then $p'$ is not in $S$, so these Voronoi cells are not both red and $q$ and $q'$ cannot both be red. 

To finish the proof, we need to show that with positive probability, there is no blue copy of $K$. Observe that since the points of $K$ have distance at most $R - 1$ from each other, if there is a blue copy of $K$ in the coloring of $\mathbb{E}^n$, then we already have a blue copy in the axis-aligned box with one corner at the origin and side length $3R$. This box contains $3^n|P|$ Voronoi cells, which we label $U_1,\ldots,U_{3^n|P|}$. 

Let $K'$ be a maximum subset of $K$ which is $5$-separated, so $|K'| \geq 11^{-n} |K|$ by Lemma~\ref{sep}. Denote the points of $K'$ by $K'=\{k_0, k_1,\ldots,k_{|K'| - 1}\}$, where we may assume that $k_0$ is the origin and $k_1,\ldots,k_d$ with $d \leq n$ form a basis for the vector space spanned by $K'$, so each element of $K'$ is a linear combination of $k_1,\ldots,k_d$. It suffices to show that with positive probability there is no blue copy of $K'$. For a map $f: \{0,1, \dots, |K'|-1\} \rightarrow \{1, 2, \dots, 3^n|P|\}$, consider the bad event $B_f$ that there is a blue copy of $K'$ with the copy of $k_i$ in $U_{f(i)}$. As each pair of points from $K'$ have distance at least $5$, the Voronoi cells $V_p$ and $V_{p'}$ that they map to under an isometry have centers of distance at least $5-2/3 = 13/3> 2 \cdot 5/3$ apart. Moreover, since $K$ has diameter at most $R - 1$, the centers have distance at most $R - 1 + 2/3 < R$, so points from two copies of the same cell are never used. Hence, the probability that $p$ and $p'$ are in $S$ are independent. Therefore, for any fixed $f$, the probability that $B_f$ happens is at most $(1-x/2)^{|K'|}<e^{-x|K'|/2}$. 

We next estimate the number of bad events $B_f$ that are realizable. That is, the number of $f$ for which there is a copy of $K'$ with the copy of $k_i$ in $U_{f(i)}$. Given a copy of $K'$ in $\mathbb{E}^n$ where $k_i$ maps to $g(i) \in \mathbb{E}^n$ for each $i$, we map the copy of $K'$ to the point $(g(0), g(1),\ldots,g(d)) \in \mathbb{E}^{(d+1)n}$. This is an injective map from the copies of $K'$ in $\mathbb{E}^n$ to $\mathbb{E}^{(d+1)n}$ since the copy of $K'$ is determined by which points $k_0,k_1,\ldots,k_d$ map to. 

Let $U$ be one of the Voronoi cells $U_1,\ldots,U_{3^n|P|}$, with center $p$. The Voronoi cell $U$ is given as the intersection of half-spaces $H_{pp'}$ which contain $p$ and whose boundary is the hyperplane equidistant from $p$ and $p'$. By Lemma \ref{halfspace}, there are at most $5^n$ such half-spaces. The linear inequality defining whether a point $(x_1,\ldots,x_n)$ is in a half-space $H$ is of the form $a_1x_1+\cdots+a_nx_n \leq b$ for some $a_1,\ldots,a_n,b \in \mathbb{E}^1$. As $k_i$ is a linear combination of $k_1, \dots, k_d$, these observations show that, for any copy of $K'$ in $\mathbb{E}^n$ and any $i$ and $j$, we can determine whether $k_i$ is mapped into $U_j$ by considering a system of at most $5^n$ linear inequalities in the $(d+1)n$ coordinates of the point $(g(0), g(1),\ldots,g(d)) \in \mathbb{E}^{(d+1)n}$ that $K'$ is mapped to. Since the number of pairs $(i,j)$ is $|K'| \cdot 3^n|P|$, we can therefore tell which $B_f$ are feasible (i.e., which mappings of the points of $K'$ to Voronoi cells are actually realizable by a copy of $K'$) by the sign patterns of a sequence of $5^n |K'| \cdot 3^n|P|$ linear equations. 

We can now bound the number of feasible $B_f$ by using an appropriate version of the Milnor--Thom  theorem~\cite{Milnor,OP,Thom}. For a discussion of this theorem and its history, as well as the statement we present below, we refer the interested reader to Section 6.2 of Matou\v sek's book~\cite{Mat}.

\begin{theorem}
For $M \geq N \geq 2$, the number of sign patterns of $M$ polynomials in $N$ variables, each of degree at most $D$, is at most $\left(\frac{50DM}{N}\right)^N$. 
\end{theorem}

Taking $D=1$, $M = 5^n |K'| \cdot 3^n|P| \leq |K'|(60n^{1/2}R)^n$ and $N = (d+1)n$, we see that the number of feasible bad events $B_f$ is at most 
$$\left(\frac{50DM}{(d+1)n}\right)^{(d+1)n}\leq \left(50|K'|(60n^{1/2}R)^n\right)^{2 n^2} \leq e^{2n^2\ln(50 |K'|) + 2n^3\ln(60 n^{1/2} R)}.$$
Therefore, since each event $B_f$ holds with probability at most $e^{-x|K'|/2}$, we see that as long as $x|K'|/2>2n^2\ln(50 |K'|) + 2n^3\ln(60 n^{1/2} R)$, then, with positive probability, the desired coloring exists. By using $x=20^{-n}$, $|K'| \geq 11^{-n}|K|$ and $|K| \geq 10^{4n} \log R$, one may verify that $x|K'|/4 > 2n^2\ln(50 |K'|)$ and $x|K'|/4 > 2n^3\ln(60 n^{1/2} R)$, completing the proof.

\section{Concluding remarks}

Let us say that a set $X  \subset \mathbb{E}^d$ is \emph{$f$-Ramsey} for a function $f : \mathbb{N} \rightarrow \mathbb{N}$ if any coloring of $\mathbb{E}^n$, $n \geq d$, with at most $f(n)$ colors contains a monochromatic copy of $X$. For instance, the result of Frankl and Wilson~\cite{FW} used to prove Theorem~\ref{thm:FW} was the statement that $\ell_2$ is $2^{c'n}$-Ramsey. By substituting any $f$-Ramsey set $X$ for $\ell_2$ in the proof of that theorem, we can easily deduce the following result.

\begin{theorem} \label{thm:super}
For any $f$-Ramsey set $X$, $\mathbb{E}^n \longrightarrow (X,K)$ for any set $K \subset \mathbb{E}^n$ of size at most $f(n)$.
\end{theorem}

When $X$ is a rectangular parallelepiped or a non-degenerate simplex, results of Frankl and R\"odl~\cite{FR} show that one may take $f(n) = 2^{c'n}$, where $c'$ may depend on the given configuration $X$. For all such $X$, Theorem~\ref{thm:main} easily implies that the estimate on the size of $K$ in Theorem~\ref{thm:super} is best possible up to the constant $c'$ in the exponent. 

Following~\cite{EGMRSS1}, we say that a set $X \subset \mathbb{E}^d$ is \emph{Ramsey} if it is $f$-Ramsey for some function $f : \mathbb{N} \rightarrow \mathbb{N}$ with the property that $f(n) \rightarrow \infty$ as $n \rightarrow \infty$. Theorem~\ref{thm:super} then says that for any Ramsey set $X$ and any finite set $K \subset \mathbb{E}^m$, there exists $n$ such that $\mathbb{E}^n \longrightarrow (X,K)$. In particular, by a beautiful result of K\v r\'i\v z~\cite{K91}, this holds when $X$ is a regular polygon. The following result gives a converse.

\begin{theorem}
Assuming the axiom of choice, if $X \subset \mathbb{E}^d$ is a finite set which is not Ramsey, there exists a natural number $m$ and a finite set $K \subset \mathbb{E}^m$ such that $\mathbb{E}^n \centernot\longrightarrow (X, K)$ for all $n$.
\end{theorem}

\begin{proof}
Since $X$ is not Ramsey, there exists a least natural number $r$ such that $\mathbb{E}^n \overset{r}{\centernot\longrightarrow} X$ for all $n$. By the minimality of $r$, there is an $m$ such that every $(r-1)$-coloring of $\mathbb{E}^m$ contains a monochromatic copy of $X$. But then, by the De Bruijn--Erd\H{o}s theorem (and it is here that we invoke the axiom of choice), there must be a finite set $K \subset \mathbb{E}^m$ such that every $(r-1)$-coloring of $K$ contains a monochromatic copy of $X$. 

Suppose now that $\chi : \mathbb{E}^n \rightarrow \{1, 2, \dots, r\}$ is an $r$-coloring of $\mathbb{E}^n$ containing no monochromatic copy of $X$. We claim that the red/blue-coloring of $\mathbb{E}^n$ where a point is colored red if it received color $1$ under $\chi$ and blue otherwise contains no red copy of $X$ and no blue copy of $K$. Indeed, a red copy of $X$ would yield a copy of $X$ in color $1$ under $\chi$, while a blue copy of $K$ would yield an $(r-1)$-colored copy of $K$ under $\chi$, which, by the choice of $K$, would contain a monochromatic copy of $X$. In either case, this would contradict the definition of $\chi$.
\end{proof}

Being more particular about our choice of $K$, we were unable to decide whether, for every natural number $m$, there exists a natural number $n$ such that $\mathbb{E}^n \longrightarrow (\ell_3, \ell_m)$. It seems unlikely that this holds for large $m$, but we were at a loss to exhibit a coloring which confirms our suspicion. A first step in the right direction was made by Erd\H{o}s et al.~\cite{EGMRSS1}, who showed that $\mathbb{E}^n \centernot\longrightarrow (\ell_6, \ell_6)$ for all $n$.

As a final note, we mention another problem of Erd\H{o}s et al.~\cite{EGMRSS2}: for any natural number $n$, does there exist a natural number $m$ depending only on $n$ such that for every set $K \subset \mathbb{E}^n$ of size $m$ there is a two-coloring of $\mathbb{E}^n$ with no monochromatic copy of $K$? By rescaling, we may assume that the smallest distance between any two points in $K$ is equal to one and, therefore, that $\ell_2 \subset K$ and $K$ is $1$-separated. 
Theorem~\ref{thm:main} then implies that if the diameter of $K$ is at most $R - 1$ and $m \geq 10^{4n} \log R$, there is indeed a two-coloring of $\mathbb{E}^n$ with no monochromatic copy of $K$. This partially answers the question of Erd\H{o}s et al.~and a complete answer would follow if we could remove the dependence on $R$ in Theorem~\ref{thm:main} (a problem which is also interesting in its own right). 

\vspace{3mm}
{\bf Acknowledgements.} This paper was written while both authors were visiting the Simons Institute for the Theory of Computing in Berkeley and we are grateful for their generous support. The authors would also like to thank Noga Alon and Ben Green for helpful discussions. Finally, we wish to thank David Ellis, Ron Graham and an anonymous referee for a number of useful comments and corrections. In particular, the anonymous referee was the one to point us to the paper by Szlam~\cite{Sz01}, helping us to greatly improve the results in the concluding remarks.

\end{document}